\documentclass{scrartcl}

\usepackage{amssymb,amsmath,amsfonts,amsthm,authblk,enumerate}
\usepackage{tikz}
\usepackage{mathrsfs,hyperref}

\usetikzlibrary{decorations.text}
\usetikzlibrary{decorations}
\usetikzlibrary{snakes}
\tikzset{
    position/.style args={#1:#2 from #3}{
        at=(#3.#1), anchor=#1+180, shift=(#1:#2)
    }
}

\theoremstyle{plain}
\newtheorem{Definition}{Definition}
\numberwithin{Definition}{section}

\theoremstyle{plain}
\newtheorem{Theorem}[Definition]{Theorem}
\newtheorem{Lemma}[Definition]{Lemma}

\theoremstyle{definition}

\newcommand{\N}{\mathbb N}

\title{On Chordal Graph and Line Graph Squares}
\author[1]{Robert Scheidweiler}
\author[2]{Sebastian Wiederrecht}
\affil[2]{TU Berlin}
\affil[1]{RWTH Aachen University}
\date{ }

\begin{document}

\maketitle

\setcounter{page}{1}

 \begin{abstract}
 In this work we investigate the chordality of squares and line graph squares of graphs. We prove a sufficient condition for the chordality of squares of graphs not containing induced cycles of length at least five. Moreover, we characterize the chordality of graph squares by forbidden subgraphs. Transferring that result to line graphs allows us to characterize the chordality of line graph squares.\\
 
 \end{abstract}

\section{Introduction}
A graph $G=(V,E)$ is said to be chordal if every cycle $C$ of length $n\geq 4$ in $G$ has a chord, i.e., there is an edge $e\in E$ connecting two nonconsecutive vertices of the cycle. Here, as usual, a cycle of length $n$ is defined as alternating vertex edge sequence $(v_1,e_1,v_2,e_2, \dots,v_n, e_n)$ such that $e_{i}\cap e_{i+1 \mod n}=\{v_i\}.$ Due to their strong combinatorial properties, chordal graphs are one of the most extensively studied graph classes in Graph Theory and Discrete Optimization. In order to convey the power of chordality respectively the properties of chordal graphs to other graph classes it seems natural to consider different types of graph transformations and their impact on chordality. Here, we investigate graph powers (mostly squares) and line graphs. 
Let in the following always a graph $G=(V,E)$ be given. We define the $k$-th graph power of $G$ by
\[
G^k:=\left( V,\{xy \mid x,y\in V \text{ and }1\leq\text{dist}_{G}(x,y)\leq k \}\right).
\]
$\text{dist}_{G}(x,y)$ denotes the length of a shortest path between $x $ and $y.$
Now, obvious questions concerning graph powers and chordality are: 
\begin{itemize}
\item Is the $k$-th power $G^k$ of a graph $G$ chordal?
\item Is the $k$-th power $G^k$ of a chordal graph $G$ chordal?  
\end{itemize}

Both questions were already discussed in the literature and in Section \ref{Sec:chordalsquares} we summarize several results which are strongly related to our work. After that, we will give a sufficient condition for the chordality of a graph square and characterize the class of graphs with chordal squares. 
In Section \ref{Sec:linegraphsquares} we characterize the chordality of line graph squares. The line graph of a graph $G$ is given by $$L(G):=(E, \{ef\mid e \text{ and } f \text{ are adjacent}  \}).$$
Clearly, these results have also an impact on algorithms. Namely, we exhibit the classes of graphs for which chordal optimization algorithms work when applied to the square and line graph square. Therefore, e.g., colorings with distance conditions such as strong edge colorings can be found in polynomial time in these classes.

\section{Chordal Squares of Graphs}\label{Sec:chordalsquares}

A first and rather easy to find subclass of graphs with chordal squares is the class of graphs whose components form complete graphs when squared. These graphs can easily be identified using their diameter. The diameter of a connected graph $G$ is given by $$\operatorname{dm}(G):=\max_{v,w\in V(G)}\operatorname{dist}_G(v,w),$$ while the diameter of a general graph is given by the maximum diameter of its components. It is fairly easy to see that all components of $G^{\operatorname{dm}(G)}$ are complete and it is straight forward to show that $\operatorname{dm}(G)$ is in fact the smallest power for which all components of $G$ become complete.

\begin{Lemma}\label{thm5.28}
All components of $G^k$ are complete graphs if and only if $k\geq\operatorname{dm}(G)$.	
\end{Lemma}

%\begin{proof}
%It suffices to assume $G$ to be connected, since in any non-connected graph there is at least one component realizing its diameter. So let $G$ be a connected graph with diameter $\operatorname{dm}(G)=k$. Let $x,y\in V(G)$ be two vertices of $G$ realizing the diameter, so $\operatorname{dist}_G(x,y)=k$, at least one pair of such vertices must exist, otherwise $k$ would not be the diameter.\\
%Now, for any $k'<k$ the vertices $x$ and $y$ cannot be adjacent in $G^{k'}$ and thus the graph $G^{k'}$ is complete if and only if $\operatorname{dm}(G)\leq k'$.
%\end{proof} 

In general, chordal graphs are not closed under the taking of powers (compare Figure \ref{fig3.2} a) ). Nevertheless it is possible to show results for odd graph powers. In 1980 Laskar and Shier (see \cite{LaskarShier1980chordal}) showed that $G^3$ and $G^5$ are chordal if $G$ is chordal and they conjectured that every odd power of a chordal graph is chordal. Duchet (see \cite{Duchet1984classical}) proved an even stronger result which led to a number of so called Duchet-type results on graph powers.

\begin{Theorem}\cite{Duchet1984classical}\label{thm3.12}
Let $k\in\N$. If $G^k$ is chordal, so is $G^{k+2}$.
\end{Theorem}

\begin{figure}
\begin{center}
\begin{tikzpicture}

\node (center) [inner sep=1.5pt] {};

\node (label) [inner sep=1.5pt,position=135:1.7cm from center] {};

\node (u1) [inner sep=1.5pt,position=90:1.6cm from center,draw,circle,fill] {};
\node (u2) [inner sep=1.5pt,position=162:1.6cm from center,draw,circle,fill] {};
\node (u3) [inner sep=1.5pt,position=234:1.6cm from center,draw,circle,fill] {};
\node (u4) [inner sep=1.5pt,position=306:1.6cm from center,draw,circle,fill] {};
\node (u5) [inner sep=1.5pt,position=18:1.6cm from center,draw,circle,fill] {};

\node (v1) [inner sep=1.5pt,position=126:0.8cm from center,draw,circle] {};
\node (v2) [inner sep=1.5pt,position=198:0.8cm from center,draw,circle] {};
\node (v3) [inner sep=1.5pt,position=270:0.8cm from center,draw,circle] {};
\node (v4) [inner sep=1.5pt,position=342:0.8cm from center,draw,circle] {};
\node (v5) [inner sep=1.5pt,position=54:0.8cm from center,draw,circle] {};

\node (lu1) [position=90:0.07cm from u1] {$u_1$};
\node (lu2) [position=162:0.07cm from u2] {$u_2$};
\node (lu3) [position=234:0.07cm from u3] {$u_3$};
\node (lu4) [position=306:0.07cm from u4] {$u_4$};
\node (lu5) [position=18:0.07cm from u5] {$u_5$};

\node (lv1) [position=126:0.07cm from v1] {$w_1$};
\node (lv2) [position=198:0.07cm from v2] {$w_2$};
\node (lv3) [position=270:0.07cm from v3] {$w_3$};
\node (lv4) [position=342:0.07cm from v4] {$w_4$};
\node (lv5) [position=54:0.07cm from v5] {$w_5$};

\path
(u1) edge (v1)
	 edge (v5)
(u2) edge (v1)
	 edge (v2)
(u3) edge (v2)
	 edge (v3)
(u4) edge (v3)
	 edge (v4)
(u5) edge (v4)
	 edge (v5)
;

\path[bend right] 
(v1) edge (v2)
(v2) edge (v3)
(v3) edge (v4)
(v4) edge (v5)
(v5) edge (v1)
;

\path
(v2) edge (v4)
(v4) edge (v1)
;

\node (glabel) [inner sep=0pt,position=270:30mm from center] {a)};

\node (center2) [inner sep=1.5pt,position=0:6.2cm from center] {};
		
\node (label) [inner sep=1.5pt,position=135:1.7cm from center2] {};
		
\node (u1) [inner sep=1.5pt,position=90:1.8cm from center2,draw,circle,fill] {};
\node (u2) [inner sep=1.5pt,position=180:1.8cm from center2,draw,circle,fill] {};
\node (u3) [inner sep=1.5pt,position=270:1.8cm from center2,draw,circle,fill] {};
\node (u4) [inner sep=1.5pt,position=0:1.8cm from center2,draw,circle,fill] {};
		
\node (v1) [inner sep=1.5pt,position=135:0.9cm from center2,draw,circle] {};
\node (v2) [inner sep=1.5pt,position=225:0.9cm from center2,draw,circle] {};
\node (v3) [inner sep=1.5pt,position=315:0.9cm from center2,draw,circle] {};
\node (v4) [inner sep=1.5pt,position=45:0.9cm from center2,draw,circle] {};

\node (lu1) [position=90:0.07cm from u1] {$u_1$};
\node (lu2) [position=180:0.07cm from u2] {$u_2$};
\node (lu3) [position=270:0.07cm from u3] {$u_3$};
\node (lu4) [position=0:0.07cm from u4] {$u_4$};
		
\node (lv1) [position=135:0.07cm from v1] {$w_1$};
\node (lv2) [position=225:0.07cm from v2] {$w_2$};
\node (lv3) [position=315:0.07cm from v3] {$w_3$};
\node (lv4) [position=45:0.07cm from v4] {$w_4$};
		
\path
(u1) edge (v1)
	 edge (v4)
(u2) edge (v1)
	 edge (v2)
(u3) edge (v2)
	 edge (v3)
(u4) edge (v3)
	 edge (v4)
;
		
\path [bend right]
(v1) edge (v2)
(v2) edge (v3)
(v3) edge (v4)
(v4) edge (v1)
;

\node (glabel2) [inner sep=0pt,position=270:30mm from center2] {b)};

\end{tikzpicture}
\caption{a) Example of a chordal graph whose square is not chordal.	\newline b) The graph $F_4$ - a "nonchordal sunflower" of size 4.}
\label{fig3.2}
\end{center}
\end{figure}
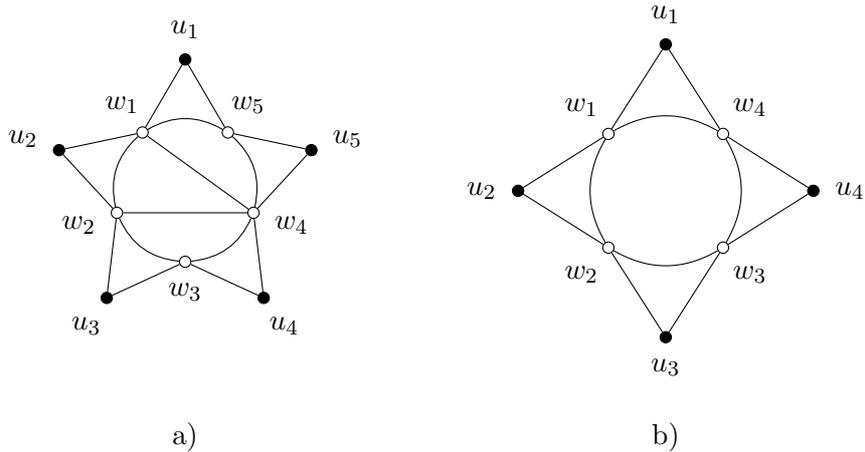

Duchet's result was utilized by Laskar and Shier to characterize the actual class of chordal graphs, closed under the taking of powers, by finding those chordal graphs whose squares would also be chordal.
In order to state their result, we need some definitions. For a subset $U\subseteq V$ we define $$G[U]:=\left(U, \{e\in E\mid e\subset U\}\right)$$ the subgraph induced by $U.$ With this notion we can rephrase the definition of chordality: $G$ is chordal if and only if it does not contain an induced cycle, i.e., there is no subset $U\subseteq V$ such that the graph $G[U]$ is isomorphic to a cycle $C$ of length $l\geq 4.$  

\begin{Definition}[Sunflower]\label{def:sunflower}
A {\em sunflower} of size $n$ is a graph $S=\left( U\cup W, E \right)$ with $U=\{u_1,\dots,u_n\}$ and $W=\{w_1,\dots,w_n\}$ such that 
\begin{itemize}
\item $G[W]$ is chordal, 
\item $U$ is a stable set, i.e., a set of pairwise nonadjacent vertices, and 
\item $u_iv_j\in E$ if and only if $j=i$ or $j=i+1\left( \!\!\!\mod n\right)$.
\end{itemize}
The family of all sunflowers of size $n$ is denoted by $\mathcal{S}_n$. A sunflower $S\in\mathcal{S}_n$ contained in some graph $G$ is called {\em suspended} if there exists a vertex $v\notin V(S)$, such that $v$ is adjacent to at least one pair of vertices $u_i$ and $u_j$ with $j\neq i\pm1\left( \!\!\!\mod n\right)$. Otherwise we say $S$ is unsuspended. (Compare Figure \ref{fig3.3}.)
\end{Definition}

\begin{figure}
\begin{center}
\begin{tikzpicture}

\node (center) [inner sep=0pt] {};

\node (1) [draw,circle,inner sep=1.5pt,fill,position=90:1.6cm from center] {};
\node (l1) [position=90:0.07cm from 1] {$w_1$};

\node (2) [draw,circle,inner sep=1.5pt,fill,position=141.4:1.6cm from center] {};
\node (l2) [position=141.4:0.07cm from 2] {$w_2$};

\node (3) [draw,circle,inner sep=1.5pt,fill,position=192.8:1.6cm from center] {};
\node (l3) [position=192:0.07cm from 3] {$w_3$};

\node (4) [draw,circle,inner sep=1.5pt,fill,position=244.2:1.6cm from center] {};
\node (l4) [position=244.2:0.07cm from 4] {$w_4$};

\node (5) [draw,circle,inner sep=1.5pt,fill,position=295.6:1.6cm from center] {};
\node (l5) [position=295.6:0.07cm from 5] {$w_5$};

\node (6) [draw,circle,inner sep=1.5pt,fill,position=347:1.6cm from center] {};
\node (l6) [position=347:0.07cm from 6] {$w_6$};

\node (7) [draw,circle,inner sep=1.5pt,fill,position=38.5:1.6cm from center] {};
\node (l7) [position=38.5:0.07cm from 7] {$w_7$};

\node (u1) [draw,circle,inner sep=1.5pt,fill,position=63.8:2.35cm from center] {};
\node (lu1) [position=63.8:0.07cm from u1] {$u_1$};

\node (u2) [draw,circle,inner sep=1.5pt,fill,position=115.2:2.35cm from center] {};
\node (lu2) [position=115.2:0.07cm from u2] {$u_2$};

\node (u3) [draw,circle,inner sep=1.5pt,fill,position=166.6:2.35cm from center] {};
\node (lu3) [position=166.6:0.07cm from u3] {$u_3$};

\node (u4) [draw,circle,inner sep=1.5pt,fill,position=218:2.35cm from center] {};
\node (lu4) [position=218:0.07cm from u4] {$u_4$};

\node (u5) [draw,circle,inner sep=1.5pt,fill,position=270:2.35cm from center] {};
\node (lu5) [position=270:0.07cm from u5] {$u_5$};

\node (u6) [draw,circle,inner sep=1.5pt,fill,position=320.8:2.35cm from center] {};
\node (lu6) [position=320.8:0.07cm from u6] {$u_6$};

\node (u7) [draw,circle,inner sep=1.5pt,fill,position=12.3:2.35cm from center] {};
\node (lu7) [position=12.3:0.07cm from u7] {$u_7$};

\path
(1) edge [bend right,bend angle=5] (2)
(2) edge [bend right,bend angle=5] (3)
(3) edge [bend right,bend angle=5] (4)
(4) edge [bend right,bend angle=5] (5)
(5) edge [bend right,bend angle=5] (6)
(6) edge [bend right,bend angle=5] (7)
(7) edge [bend right,bend angle=5] (1)
(u1) edge (1)
	 edge (7)
(u2) edge (1)
	 edge (2)
(u3) edge (2)
	 edge (3)
(u4) edge (3)
	 edge (4)
(u5) edge (4)
	 edge (5)
(u6) edge (5)
	 edge (6)
(u7) edge (6)
	 edge (7)
(3) edge (1)
(4) edge (1)
(4) edge (7)
(4) edge (6)	 
;

\node (w) [draw,circle,inner sep=1.5pt,fill,position=270:4cm from center] {};
\node (lw) [position=270:0.07 from w] {$v$};

\path
(w) edge [dotted] (u4)
(w) edge [dotted] (u6);

\end{tikzpicture}
\caption{A suspended sunflower $S\in\mathcal{S}_7$.}\label{fig3.3}
\end{center}
\end{figure}
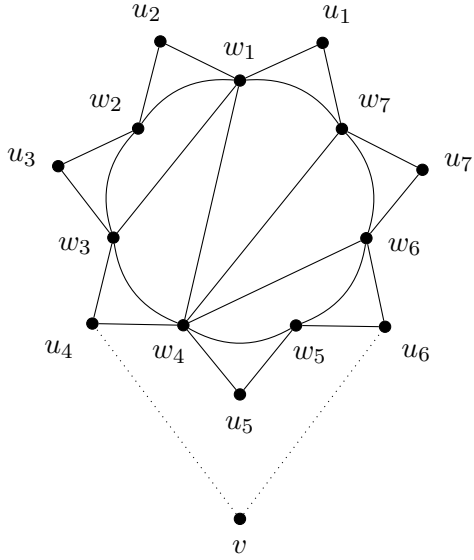

Now, we are ready to state Laskar's and Shier's result:

\begin{Theorem}\cite{LaskarShier1983powercenterchordal}\label{thm3.13}
Let $G$ be chordal. Then $G^2$ is chordal if and only if $G$ does not contain an unsuspended sunflower $S\in\mathcal{S}_n$ with $n\geq4$.
\end{Theorem}

Combining the foregoing result with Duchet's Theorem leads to the family of chordal graphs that is closed under taking powers.

\begin{Theorem} \cite{LaskarShier1983powercenterchordal}\label{thm3.14}
Let $G$ be chordal without unsuspended sunflowers $S\in\mathcal{S}_n$ with $n\geq4$, then $G^k$ is chordal and does not contain any unsuspended sunflower $S\in\mathcal{S}_n$ with $n\geq4$ for all $k\in\N$.
\end{Theorem}

In the following, we investigate when the square of a graph contains chordless cycles of length $l\geq 4.$ We begin with a lemma which will be extensively used in the remainder of this work. It is a slight generalization of a result from \cite{LaskarShier1980chordal}.

\begin{Lemma}\label{lem:notwoconsec}
Let $k\in\mathbb{N}$ with $k\geq2.$ If $C$ is a chordless cycle of length $l\geq 4$ in $G^k$, then $G^r$ cannot contain two consecutive edges of $C$ for all $r\leq\lfloor{\frac{k}{2}}\rfloor.$ In particular, $G^r$ contains at most $\lfloor{\frac{l}{2}}\rfloor$ edges of $C$ for all $r\leq\lfloor{\frac{k}{2}}\rfloor$.
\end{Lemma}

\begin{proof}
Let $C=\left( v_1e_1\dots v_le_l\right)$ be a chordless cycle in $G^k$ of length $l\geq 4.$ Suppose there is an $1\leq r\leq\lfloor{\frac{k}{2}}\rfloor$ and an $i\in\{1,\dots,l\}$ with $\{e_i,e_{i+1\left( \!\!\!\mod l\right)}\}\subseteq E({G^r})$.
Then $$\displaystyle \text{dist}_{G}({v_{i+j\left( \!\!\!\mod l\right)}},{v_{i+1+j\left( \!\!\!\mod l\right)}})\leq r\leq\left\lfloor{\frac{k}{2}}\right\rfloor\text{ for } j=0,1$$ and, therefore, $\text{dist}_{G}({v_i},{v_{i+2\left( \!\!\!\mod l\right)}})\leq k$. Hence, the edge $v_iv_{i+2\left( \!\!\!\mod l\right)}$ exists in $G^k.$ It is a chord of $C$ in $G^k$ which is impossible.  
\end{proof}

The search for induced subgraphs that are responsible for $G^2$ not being chordal was started by Balakrishnan and Paulraja in \cite{balakrishnan1981graphsa} and \cite{balakrishnan1981graphsb}. In these papers a path on five vertices $v_1,\dots,v_5$ with added edge $v_2v_4$ is denoted by $P_5+a.$ As usual, $K_n$ is a complete graph on $n$ vertices and $K_{x,y}$ a complete bipartite (2-colorable) graph in which $x$ and $y$ are the sizes of the color classes. The special graph $K_{1,3}$ is called claw.

\begin{Theorem}\cite{balakrishnan1981graphsa}, \cite{balakrishnan1981graphsb}\label{thm3.15}
If $G$ does not contain an induced $K_{1,3}$, an induced $P_5+a$ nor any induced cycle $C$ of length $n\geq6$, then $G^2$ is chordal.	
\end{Theorem}

Please note, that such a $P_5+a$ is part of the graph $F_4$ shown in Figure \ref{fig3.2} b). Flotow (see \cite{flotow1997graphs}) proved that there is no finite family of forbidden subgraphs which ensures the chordality of $G^m$ for $m\geq 2$. He assumed that such a family $\left( G_i\right)_{i\in I}$ of subgraphs would exist and constructed a graph $G$ by taking all those $G_i$ and joining all their vertices to an additional vertex $v^*$. This construction made $G^m$ complete and therefore chordal for all $m\geq 2.$ Therefore, it is impossible to find such a finite family of subgraphs. 

The next theorem is also due to Flotow.

\begin{Theorem}\cite{flotow1997graphs}\label{thm3.16}
If $G$ is chordal and does not contain an induced claw, then $G^2$ is chordal.	
\end{Theorem}

There seems to be some gap between graphs containing no induced cycles of length at least $4$ and of length at least $6$. The next Theorem will show that the graph $F_4$ depicted in Figure \ref{fig3.2} b) fills this gap. In its proof we will need the notion of the neighborhood of some vertex $v$ denoted by $N_G(v):=\{u\in V\mid uv \in E \}.$ 

\begin{Theorem}\label{thm:chordalsq}
If $G$
\begin{itemize}
\item does not contain an induced claw,
\item nor an induced cycle $C$ of length $l\geq5$ and 
\item any induced $F_4$ in $G$ is suspended, i.e., there exists a vertex $v$ such that the edges $vu_1$ and $vu_3$ both exist or $vu_2$ and $vu_4$ both exist, 
\end{itemize}
then $G^2$ is chordal.
\end{Theorem}

\begin{proof}
Let $G$ be given as above and suppose that $G^2$ is not chordal. Of course, $G^2$ contains a chordless cycle $C_{G^2}$ of length $l\geq 4$ with vertex sequence $\left( u_1\dots u_{l}\right).$ The proof is divided into two major cases. Each of them is again divided into subcases as follows:
 \begin{enumerate}
	\item[1.] Case: $l\geq 5.$
		\begin{enumerate}
		\item[] \begin{enumerate}
				
			\item [Subcase 1.1:] No edge of $C_{G^2}$ exists in $G$.
			\item [Subcase 1.2:] At least one edge of $C_{G^2}$ already exists in $G$.
		\end{enumerate}
		\end{enumerate}

	\item [2.] Case: $l=4.$
		\begin{enumerate}
		\item[] \begin{enumerate}
				
			\item [Subcase 2.1:] The edges $u_1u_2$ and $u_3u_4$ exist in $G$.
			\item [Subcase 2.2:] Just the edge $u_1u_2$ is contained in $G$.
			\item [Subcase 2.3:] $U=\{u_1,u_2,u_3,u_4\}$ is a stable set in $G,$ i.e., the vertices are pairwise nonadjacent.
		\end{enumerate}
		\end{enumerate}
\end{enumerate}

Subcase 1.1: Suppose $l\geq 5$ and no edge of $C_{G^2}$ is contained in $G$.
Hence the set $U=\{u_1,\dots,u_l\}$ of the vertices forming $C_{G^2}$ in $G^2$ forms a stable set in $G$. With $u_i$ and $u_{i\pm1\left( \!\!\!\mod c\right)}$ being adjacent in $G^2$ there exists paths of length $2$ between them in $G$ and therefore a set $W=\{w_1,\dots,w_l\}$ of additional vertices exists. The vertices in $W$ have to be disjoint and $N_{G}({w_i})\cap U=\{u_i,u_{i+1\left( \!\!\!\mod l\right)}\}$ holds for all $i\in\{1,\dots,l\}$, otherwise $C_{G^2}$ would have a chord.
This leads to a cycle $C_{G}$ of length $2l$ in $G$, alternating between vertices from $U$ and $W$. Since $l\geq 5$ this cycle must contain a chord. By the case assumption, a chord in $C_{G}$ must be of the form $w_iw_j$. If in addition $j= i\pm 1\left( \!\!\!\mod l\right)$ holds, then such a chord shortens $C_{G}$ just by $1$.

\begin{figure}
	\begin{center}
		\begin{tikzpicture}
		
		\node (center) [inner sep=1.5pt] {};
		
		\node (label) [inner sep=1.5pt,position=135:1.7cm from center] {};
		
		\node (u1) [inner sep=1.5pt,position=90:1.6cm from center,draw,circle,fill] {};
		\node (u2) [inner sep=1.5pt,position=162:1.6cm from center,draw,circle,fill] {};
		\node (u3) [inner sep=1.5pt,position=234:1.6cm from center,draw,circle,fill] {};
		\node (u4) [inner sep=1.5pt,position=306:1.6cm from center,draw,circle,fill] {};
		\node (u5) [inner sep=1.5pt,position=18:1.6cm from center,draw,circle,fill] {};
		
		\node (v1) [inner sep=1.5pt,position=126:0.8cm from center,draw,circle] {};
		\node (v2) [inner sep=1.5pt,position=198:0.8cm from center,draw,circle] {};
		\node (v3) [inner sep=1.5pt,position=270:0.8cm from center,draw,circle] {};
		\node (v4) [inner sep=1.5pt,position=342:0.8cm from center,draw,circle] {};
		\node (v5) [inner sep=1.5pt,position=54:0.8cm from center,draw,circle] {};

		\node (lu1) [position=90:0.07cm from u1] {$u_1$};
		\node (lu2) [position=162:0.07cm from u2] {$u_2$};
		\node (lu3) [position=234:0.07cm from u3] {$u_3$};
		\node (lu4) [position=306:0.07cm from u4] {$u_4$};
		\node (lu5) [position=18:0.07cm from u5] {$u_5$};
		
		\node (lv1) [position=126:0.07cm from v1] {$w_1$};
		\node (lv2) [position=198:0.07cm from v2] {$w_2$};
		\node (lv3) [position=270:0.07cm from v3] {$w_3$};
		\node (lv4) [position=342:0.07cm from v4] {$w_4$};
		\node (lv5) [position=54:0.07cm from v5] {$w_5$};

		\path
		(u1) edge (v1)
		edge (v5)
		(u2) edge (v1)
		edge (v2)
		(u3) edge (v2)
		edge (v3)
		(u4) edge (v3)
		edge (v4)
		(u5) edge (v4)
		edge (v5)
		;
		
		\path[bend right] 
		(v1) edge (v2)
		(v2) edge (v3)
		(v3) edge (v4)
		(v4) edge (v5)
		(v5) edge (v1)
		;
		
		\node (glabel2) [inner sep=0pt,position=270:30mm from center] {a)};

		\node (centerhelp) [position=0:6.2cm from center] {};
		\node (center2) [inner sep=1.5pt,position=270:1.3cm from centerhelp] {};
		\node (anchor1) [inner sep=1.5pt,position=0:0.8cm from center2] {};
		\node (anchor2) [inner sep=1.5pt,position=180:0.8cm from center2] {};
		
		\node (label) [inner sep=1.5pt,position=135:1.7cm from center2] {};
		
		\node (u1) [inner sep=1.5pt,position=90:1.6cm from anchor1,draw,circle,fill] {};
		\node (u2) [inner sep=1.5pt,position=90:1.6cm from anchor2,draw,circle,fill] {};
		\node (u3) [inner sep=1.5pt,position=180:2.5cm from center2,draw,circle,fill] {};
		\node (u4) [inner sep=1.5pt,position=0:2.5cm from center2,draw,circle,fill] {};
		
		\node (v1) [inner sep=1.5pt,position=0:0.8cm from center2,draw,circle] {};
		\node (v2) [inner sep=1.5pt,position=180:0.8cm from center2,draw,circle] {};

		\node (lu1) [position=45:0.07cm from u1] {$u_1$};
		\node (lu2) [position=135:0.07cm from u2] {$u_2$};
		\node (lu3) [position=150:0.07cm from u3] {$u_3$};
		\node (lu4) [position=30:0.07cm from u4] {$u_l$};
		
		\node (lv1) [position=270:0.07cm from v1] {$w_1$};
		\node (lv2) [position=270:0.07cm from v2] {$w_2$};

		\path
		(u1) edge (v1)
		(u2) edge (u1)
		edge (v2)
		(u3) edge (v2)
		(u4) edge (v1)
		;
		
		\node (glabel2) [inner sep=0pt,position=270:30mm from centerhelp] {b)};
				
		\end{tikzpicture}
		\caption{a) Shortened cycle in the case $l=5$.\newline b) Direct surroundings of the edge $u_1u_2$ in $G$.}
		\label{fig3.5}
	\end{center}
\end{figure}
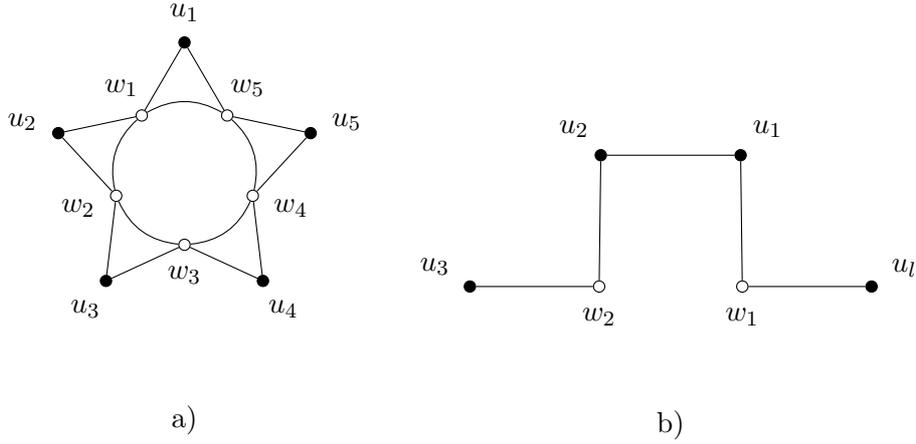 
Even if all those chords exist there is still a cycle of length $\tilde{l}\geq 5$ remaining (compare Figure \ref{fig3.5} a)). Hence, a further chord $w_iw_j$ with $j\neq i\pm1\left( \!\!\!\mod l\right)$ must exist. Therefore, we obtain ${G}[{\{w_i,w_j,u_i,u_{i+1\left(\!\!\!\mod l\right)} \}}]\cong K_{1,3}$ which is a contradiction.

Subcase 1.2:  Suppose $l\geq 5$ and at least one edge of $C_{G^2}$ is contained in $G$.
W.l.o.g. suppose the edge $u_1u_2$ exists in $G$, then by Lemma \autoref{lem:notwoconsec} the edges $u_1u_l$ and $u_2u_3$ cannot exist (compare Figure \ref{fig3.5} b)), and we need $q$ additional vertices $W=\{w_1,\dots,w_q\}.$ These vertices form paths of length $2$ between vertices in $U$ that are not adjacent in $G$ but consecutive in $C_{G^2}$. Again by Lemma \autoref{lem:notwoconsec} $q\geq\lceil{\frac{l}{2}}\rceil$ holds and a cycle $C_{G}$ of length $\tilde{l}\geq l+\lceil{\frac{l}{2}}\rceil\geq 8$ exists in $G$. 
$C_{G}$ must contain a chord. As in the last case, the possible chords join vertices in $W$ (otherwise $C_{G^2}$ contains chords). If there are only chords $w_1w_j$ or $w_2w_j$ with $j\geq 3$ the arising shorter cycles containing $w_1$ and $w_2$ have at least length $5$. Hence, the edge $w_1w_2$ must exist in $G$ and ${G}[{u_1,u_l,w_1,w_2}]\cong K_{1,3}.$

Subcase 2.1: Suppose $C_{G^2}$ has the vertex sequence $\left( u_1u_2u_3u_4\right)$ and the edges $u_1u_2$ and $u_3u_4$ exist in $G$.
Then $G$ contains two additional vertices $w_1$ and $w_2$ that form a cycle of length $6$ together with the vertices of $C_{G^2}$ (compare Figure \ref{fig3.7} a)).

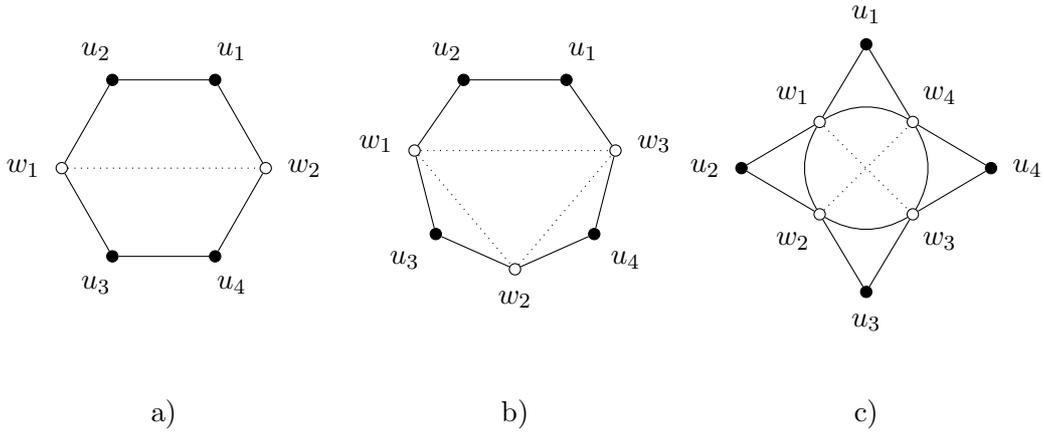
\begin{figure}[!ht]
	\begin{center}
		\begin{tikzpicture}
		
		\node (center) [inner sep=1.5pt] {};
		
		\node (label) [inner sep=1.5pt,position=135:1.7cm from center] {};
		
		\node (u1) [inner sep=1.5pt,position=60:12mm from center,draw,circle,fill] {};
		\node (u2) [inner sep=1.5pt,position=120:12mm from center,draw,circle,fill] {};
		\node (u3) [inner sep=1.5pt,position=240:12mm from center,draw,circle,fill] {};
		\node (u4) [inner sep=1.5pt,position=300:12mm from center,draw,circle,fill] {};
		
		\node (v1) [inner sep=1.5pt,position=180:12mm from center,draw,circle] {};
		\node (v2) [inner sep=1.5pt,position=0:12mm from center,draw,circle] {};

		\node (lu1) [position=60:0.07cm from u1] {$u_1$};
		\node (lu2) [position=120:0.07cm from u2] {$u_2$};
		\node (lu3) [position=240:0.07cm from u3] {$u_3$};
		\node (lu4) [position=300:0.07cm from u4] {$u_4$};
		
		\node (lv1) [position=180:0.07cm from v1] {$w_1$};
		\node (lv2) [position=0:0.07cm from v2] {$w_2$};

		\path
		(u1) edge (u2)
		edge (v2)
		(u2) edge (v1)
		(u3) edge (v1)
		edge (u4)
		(u4) edge (v2)
		;
		
		\path[dotted]
		(v1) edge (v2)
		;
		
		\node (glabel) [inner sep=0pt,position=270:30mm from center] {a)};		
		
		\node (center2) [inner sep=1.5pt,position=0:4.5cm from center] {};
		
		\node (label) [inner sep=1.5pt,position=135:1.7cm from center2] {};
		
		\node (u1) [inner sep=1.5pt,position=60:12mm from center2,draw,circle,fill] {};
		\node (u2) [inner sep=1.5pt,position=120:12mm from center2,draw,circle,fill] {};
		\node (u3) [inner sep=1.5pt,position=220:12mm from center2,draw,circle,fill] {};
		\node (u4) [inner sep=1.5pt,position=320:12mm from center2,draw,circle,fill] {};
		
		\node (v1) [inner sep=1.5pt,position=170:12mm from center2,draw,circle] {};
		\node (v2) [inner sep=1.5pt,position=270:12mm from center2,draw,circle] {};
		\node (v3) [inner sep=1.5pt,position=10:12mm from center2,draw,circle] {};

		\node (lu1) [position=60:0.07cm from u1] {$u_1$};
		\node (lu2) [position=120:0.07cm from u2] {$u_2$};
		\node (lu3) [position=220:0.07cm from u3] {$u_3$};
		\node (lu4) [position=320:0.07cm from u4] {$u_4$};
		
		\node (lv1) [position=170:0.07cm from v1] {$w_1$};
		\node (lv2) [position=270:0.07cm from v2] {$w_2$};
		\node (lv3) [position=10:0.07cm from v3] {$w_3$};

		\path
		(u1) edge (u2)
		(u2) edge (v1)
		(v1) edge (u3)
		(u3) edge (v2)
		(v2) edge (u4)
		(u4) edge (v3)
		(v3) edge (u1)
		;
		
		\path[dotted]
		(v1) edge (v3)
		(v1) edge (v2)
		(v2) edge (v3)
		;
		
		\node (glabel2) [inner sep=0pt,position=270:30mm from center2] {b)};

		\node (center3) [inner sep=1.5pt,position=0:4.5cm from center2] {};
		
		\node (label) [inner sep=1.5pt,position=135:1.7cm from center3] {};
		
		\node (u1) [inner sep=1.5pt,position=90:15mm from center3,draw,circle,fill] {};
		\node (u2) [inner sep=1.5pt,position=180:15mm from center3,draw,circle,fill] {};
		\node (u3) [inner sep=1.5pt,position=270:15mm from center3,draw,circle,fill] {};
		\node (u4) [inner sep=1.5pt,position=0:15mm from center3,draw,circle,fill] {};
		
		\node (v1) [inner sep=1.5pt,position=135:7mm from center3,draw,circle] {};
		\node (v2) [inner sep=1.5pt,position=225:7mm from center3,draw,circle] {};
		\node (v3) [inner sep=1.5pt,position=315:7mm from center3,draw,circle] {};
		\node (v4) [inner sep=1.5pt,position=45:7mm from center3,draw,circle] {};

		\node (lu1) [position=90:0.07cm from u1] {$u_1$};
		\node (lu2) [position=180:0.07cm from u2] {$u_2$};
		\node (lu3) [position=270:0.07cm from u3] {$u_3$};
		\node (lu4) [position=0:0.07cm from u4] {$u_4$};
		
		\node (lv1) [position=135:0.07cm from v1] {$w_1$};
		\node (lv2) [position=225:0.07cm from v2] {$w_2$};
		\node (lv3) [position=315:0.07cm from v3] {$w_3$};
		\node (lv4) [position=45:0.07cm from v4] {$w_4$};

		\path
		(u1) edge (v1)
		edge (v4)
		(u2) edge (v1)
		edge (v2)
		(u3) edge (v2)
		edge (v3)
		(u4) edge (v3)
		edge (v4)
		;
		
		\path [bend right]
		(v1) edge (v2)
		(v2) edge (v3)
		(v3) edge (v4)
		(v4) edge (v1)
		;
		
		\path[dotted]
		(v1) edge (v3)
		(v2) edge (v4)
		;
		
		\node (glabel3) [inner sep=0pt,position=270:30mm from center3] {c)};

		\end{tikzpicture}
		\caption{a) $C_{G}$ with just one possible chord.\newline b) $C_{G}$ with three possible chords.\newline c) $F_4$ with two forbidden chords.}
		\label{fig3.7}
	\end{center}
\end{figure}
The one possible chords would join $w_1$ and $w_2.$ Again we find a claw in $G$ which contradicts our assumption.

Subcase 2.2: Now, suppose that just the edge $u_1u_2$ exists in $G$.
Here, we get three additional vertices $w_1$, $w_2$ and $w_3$ that form a cycle $C_{G}$ of length $7$ in $G$ and we have three possible chords (compare Figure \ref{fig3.7} b)).
Since we have seen that the chord $w_1w_3$ would result in an induced $K_{1,3},$ $w_1w_3$ is forbidden, too. The chords $w_1w_2$ and $w_2w_3$ shorten the cycle just by $2$ and thus an induced cycle of length $5$ remains in $G$. This is impossible.

Subcase 2.3: Suppose $U=\{u_1,u_2,u_3,u_4\}$ is a stable set in $G.$
Therefore, we need four additional vertices $W=\{w_1,w_2,w_3,w_4\}$ to induce a cycle in $G^2$ (compare Figure \ref{fig3.7} c)). This results in another cycle $C_{G}$ in $G$ of length $8.$ We already know that the chords $w_1w_3$ and $w_2w_4$ result in induced claws.
This leaves the edges $w_1w_2$, $w_2w_3$, $w_3w_4$ and $w_4w_1$ as the only possible chords in $C_{G}$. Each of those edges shortens the cycle by $1$ and thus all four of them must exist. The resulting graph, as depicted in \autoref{fig3.7} c), is isomorphic to $F_4$ and cannot be suspended, otherwise $C_{G^2}$ would have a chord. This completes the proof.
\end{proof}

The foregoing theorem is still not a necessary condition for the square of a graph to be chordal. On the one hand excluding claws seems to be a problem because any power of a tree is chordal (and, obviously, the claw is a tree) but on the other hand allowing induced claws results in a lot of other possible graphs that produce cycles when squared. Consider the graphs of Figure \ref{fig3.9} as examples. Therefore, we consider a generalization of the sunflowers.

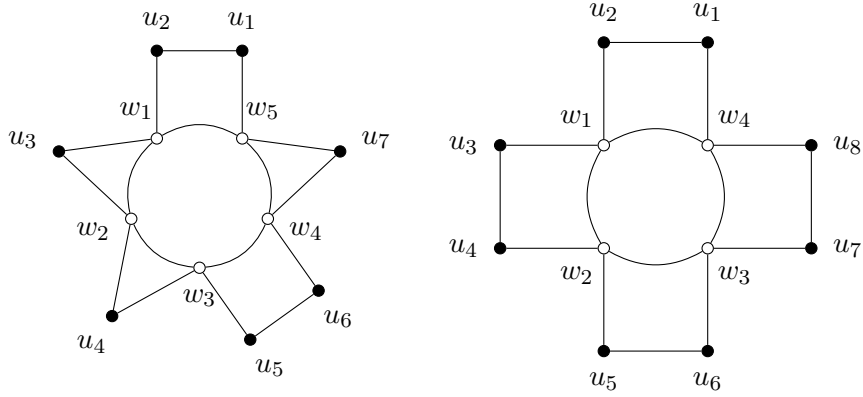
\begin{figure}[!ht]
	\begin{center}
		\begin{tikzpicture}
		
		\node (center) [inner sep=1.5pt] {};
		
		\node (v1) [inner sep=1.5pt,position=135:0.8cm from center,draw,circle] {};
		\node (v2) [inner sep=1.5pt,position=225:0.8cm from center,draw,circle] {};
		\node (v3) [inner sep=1.5pt,position=315:0.8cm from center,draw,circle] {};
		\node (v4) [inner sep=1.5pt,position=45:0.8cm from center,draw,circle] {};

		\node (u1) [inner sep=1.5pt,position=90:1.2cm from v4,draw,circle,fill] {};
		\node (u2) [inner sep=1.5pt,position=90:1.2cm from v1,draw,circle,fill] {};
		\node (u3) [inner sep=1.5pt,position=180:1.2cm from v1,draw,circle,fill] {};
		\node (u4) [inner sep=1.5pt,position=180:1.2cm from v2,draw,circle,fill] {};
		\node (u5) [inner sep=1.5pt,position=270:1.2cm from v2,draw,circle,fill] {};
		\node (u6) [inner sep=1.5pt,position=270:1.2cm from v3,draw,circle,fill] {};
		\node (u7) [inner sep=1.5pt,position=0:1.2cm from v3,draw,circle,fill] {};
		\node (u8) [inner sep=1.5pt,position=0:1.2cm from v4,draw,circle,fill] {};

		\node (lu1) [position=90:0.07cm from u1] {$u_1$};
		\node (lu2) [position=90:0.07cm from u2] {$u_2$};
		\node (lu3) [position=180:0.07cm from u3] {$u_3$};
		\node (lu4) [position=180:0.07cm from u4] {$u_4$};
		\node (lu5) [position=270:0.07cm from u5] {$u_5$};
		\node (lu6) [position=270:0.07cm from u6] {$u_6$};
		\node (lu7) [position=0:0.07cm from u7] {$u_7$};
		\node (lu8) [position=0:0.07cm from u8] {$u_8$};
		
		\node (lv1) [position=135:0.07cm from v1] {$w_1$};
		\node (lv2) [position=225:0.07cm from v2] {$w_2$};
		\node (lv3) [position=315:0.07cm from v3] {$w_3$};
		\node (lv4) [position=45:0.07cm from v4] {$w_4$};

		\path
		(u1) edge (v4)
		(u2) edge (v1)
		(u3) edge (v1)
		(u4) edge (v2)
		(u5) edge (v2)
		(u6) edge (v3)
		(u7) edge (v3)
		(u8) edge (v4)
		;
		
		\path
		(u1) edge (u2)
		(u3) edge (u4)
		(u5) edge (u6)
		(u7) edge (u8)
		;
		
		\path [bend right]
		(v1) edge (v2)
		(v2) edge (v3)
		(v3) edge (v4)
		(v4) edge (v1)
		;

		\node (center2) [inner sep=1.5pt,left of=center,node distance=6cm] {};
		
		\node (label) [inner sep=1.5pt,position=135:1.7cm from center2] {};

		\node (v1) [inner sep=1.5pt,position=126:0.8cm from center2,draw,circle] {};
		\node (v2) [inner sep=1.5pt,position=198:0.8cm from center2,draw,circle] {};
		\node (v3) [inner sep=1.5pt,position=270:0.8cm from center2,draw,circle] {};
		\node (v4) [inner sep=1.5pt,position=342:0.8cm from center2,draw,circle] {};
		\node (v5) [inner sep=1.5pt,position=54:0.8cm from center2,draw,circle] {};

		\node (u1) [inner sep=1.5pt,position=90:1cm from v5,draw,circle,fill] {};
		\node (u2) [inner sep=1.5pt,position=90:1cm from v1,draw,circle,fill] {};
		\node (u3) [inner sep=1.5pt,position=162:1.8cm from center2,draw,circle,fill] {};
		\node (u4) [inner sep=1.5pt,position=234:1.8cm from center2,draw,circle,fill] {};
		\node (u5) [inner sep=1.5pt,position=305:1cm from v3,draw,circle,fill] {};
		\node (u6) [inner sep=1.5pt,position=305:1cm from v4,draw,circle,fill] {};
		\node (u7) [inner sep=1.5pt,position=18:1.8cm from center2,draw,circle,fill] {};

		\node (lu1) [position=90:0.07cm from u1] {$u_1$};
		\node (lu2) [position=90:0.07cm from u2] {$u_2$};
		\node (lu3) [position=162:0.07cm from u3] {$u_3$};
		\node (lu4) [position=234:0.07cm from u4] {$u_4$};
		\node (lu5) [position=305:0.07cm from u5] {$u_5$};
		\node (lu6) [position=305:0.07cm from u6] {$u_6$};
		\node (lu7) [position=18:0.07cm from u7] {$u_7$};
		
		\node (lv1) [position=126:0.07cm from v1] {$w_1$};
		\node (lv2) [position=198:0.07cm from v2] {$w_2$};
		\node (lv3) [position=270:0.07cm from v3] {$w_3$};
		\node (lv4) [position=342:0.07cm from v4] {$w_4$};
		\node (lv5) [position=54:0.07cm from v5] {$w_5$};

		\path
		(u1) edge (v5)
		(u2) edge (v1)
		(u3) edge (v1)
		edge (v2)
		(u4) edge (v2)
		edge (v3)
		(u5) edge (v3)
		(u6) edge (v4)
		(u7) edge (v4)
		edge (v5)
		;
		
		\path [bend right]
		(v1) edge (v2)
		(v2) edge (v3)
		(v3) edge (v4)
		(v4) edge (v5)
		(v5) edge (v1)
		;
		
		\path
		(u1) edge (u2)
		(u5) edge (u6)
		;

		\end{tikzpicture}
	\end{center}
	\caption{Examples of graphs with nonchordal squares.}
	\label{fig3.9}
\end{figure}

\begin{Definition}[Flower]
A {\em flower} of size $n$ is a graph $F=\left( U\cup W, E \right)$ with $U=\{u_1,\dots,u_n\}$ and $W=\{w_1,\dots,w_q\}$ with $\lceil \frac{n}{2} \rceil\leq q\leq n$ satisfying the following conditions:
\begin{enumerate}[i)]
	\item There is a cycle $C$ containing all vertices of $W$ in the order $w_1,\dots,w_q$.
	
	\item The set $U=\{u_1,\dots,u_n\}$ is sorted by the appearance order of its elements along $C$ with $u_1w_q, u_2w_1\in E$ and $u_iu_j\notin E$ for $j\neq i\pm1\left(\!\!\!\mod n\right)$.
	
	\item If $w_iw_{i+1 ( \text{mod} q)}\in E(C)$, then there is exactly one $u\in U\setminus V(C)$ with ${N_F}(u)=\{w_i,w_{i+1}\}$, those vertices $u$ are called {\em pending}.
	
	\item If $w_iw_{i+1( \text{mod} q)}\notin E(C)$, then there either is one $u\in U\cap V(C)$ adjacent to $w_i$ and $w_{i+1}$, or there are exactly two vertices $u,t\in U\cap V(C)$, such that the sequence $w_iutw_{i+1}$ is part of $C$.
	
	\item The pending vertices are pairwise nonadjacent and all vertices $u\in U$ that are not pending are contained in $C$.
	
	\item There are no further edges between $U$ and $W.$
\end{enumerate}
The family of all flowers of size $n$ is denoted by $\mathcal{F}_n$.\\
If $F$ is contained in some graph $G$ and there exists an additional vertex $v$ with $vu_i, vu_j\in E(G)$ and $j\neq i\pm 1\left( \!\!\!\mod n\right)$, $F$ is called a {\em withered flower} or just {\em withered}. 	
\end{Definition}

Compare Figure \ref{fig3.9} and Figure \ref{fig3.10} for some examples of flowers. 

\begin{figure}
	\begin{center}
		\begin{tikzpicture}
		
		\node (anchor1) [] {};
		\node (anchor2) [position=0:4.7cm from anchor1] {};
		\node (anchor3) [position=0:4.7cm from anchor2] {};
		
		\node (u1) [draw,circle,fill,inner sep=1.5pt,position=90:1.6cm from anchor1] {};
		\node (u2) [draw,circle,fill,inner sep=1.5pt,position=180:1.6cm from anchor1] {};
		\node (u3) [draw,circle,fill,inner sep=1.5pt,position=270:1.6cm from anchor1] {};
		\node (u4) [draw,circle,fill,inner sep=1.5pt,position=0:1.6cm from anchor1] {};
		
		\node (w1) [draw,circle,fill,inner sep=1.5pt,position=135:0.8cm from anchor1] {};
		\node (w2) [draw,circle,fill,inner sep=1.5pt,position=225:0.8cm from anchor1] {};
		\node (w3) [draw,circle,fill,inner sep=1.5pt,position=315:0.8cm from anchor1] {};
		\node (w4) [draw,circle,fill,inner sep=1.5pt,position=45:0.8cm from anchor1] {};
		
		\node (lu1) [position=90:0.07 from u1] {$u_1$};
		\node (lu2) [position=180:0.07 from u2] {$u_2$};
		\node (lu3) [position=270:0.07 from u3] {$u_3$};
		\node (lu4) [position=0:0.07 from u4] {$u_4$};
		
		\node (lw1) [position=135:0.07cm from w1] {$w_1$};
		\node (lw2) [position=225:0.07cm from w2] {$w_2$};
		\node (lw3) [position=315:0.07cm from w3] {$w_3$};
		\node (lw4) [position=45:0.07cm from w4] {$w_4$};
		
		\node (n) [position=270:2.8cm from anchor1] {$q=4$};
		
		\path
		(u1) edge (w1)
		edge (w4)
		(u2) edge (w1)
		edge (w2)
		(u3) edge (w2)
		edge (w3)
		(u4) edge (w3)
		edge (w4)
		;
		
		\path [bend right]
		(w1) edge (w2)
		(w2) edge (w3)
		(w3) edge (w4)
		(w4) edge (w1)
		;
		
		\path
		(w2) edge (w4)
		;
		
		% % % % % % % % % % % % % % %
		
		\node (u1) [draw,circle,fill,inner sep=1.5pt,position=61.42:1cm from anchor2] {};
		\node (u2) [draw,circle,fill,inner sep=1.5pt,position=112.84:1cm from anchor2] {};
		\node (u3) [draw,circle,fill,inner sep=1.5pt,position=215.68:1cm from anchor2] {};
		\node (u4) [draw,circle,fill,inner sep=1.5pt,position=318.52:1cm from anchor2] {};
		
		\node (w1) [draw,circle,fill,inner sep=1.5pt,position=164.26:1cm from anchor2] {};
		\node (w2) [draw,circle,fill,inner sep=1.5pt,position=267.1:1cm from anchor2] {};
		\node (w3) [draw,circle,fill,inner sep=1.5pt,position=10:1cm from anchor2] {};
		
		\node (lu1) [position=61.42:0.07 from u1] {$u_1$};
		\node (lu2) [position=112.84:0.07 from u2] {$u_2$};
		\node (lu3) [position=215.68:0.07 from u3] {$u_3$};
		\node (lu4) [position=318.52:0.07 from u4] {$u_4$};
		
		\node (lw1) [position=164.26:0.07cm from w1] {$w_1$};
		\node (lw2) [position=267.1:0.07cm from w2] {$w_2$};
		\node (lw3) [position=10:0.07cm from w3] {$w_3$};
		
		\node (n) [position=270:2.8cm from anchor2] {$q=3$};
		\path
		(u1) edge (u2)
		(u2) edge (w1)
		(w1) edge (u3)
		(u3) edge (w2)
		(w2) edge (u4)
		(u4) edge (w3)
		(w3) edge (u1)
		;
		
		\path [bend left]
		(w1) edge (w2)
		edge (w3)
		;
		
		% % % % % % % % % % % % % % %
		
		\node (u1) [draw,circle,fill,inner sep=1.5pt,position=60:1cm from anchor3] {};
		\node (u2) [draw,circle,fill,inner sep=1.5pt,position=120:1cm from anchor3] {};
		\node (u3) [draw,circle,fill,inner sep=1.5pt,position=240:1cm from anchor3] {};
		\node (u4) [draw,circle,fill,inner sep=1.5pt,position=300:1cm from anchor3] {};
		
		\node (w1) [draw,circle,fill,inner sep=1.5pt,position=0:1cm from anchor3] {};
		\node (w2) [draw,circle,fill,inner sep=1.5pt,position=180:1cm from anchor3] {};
		
		\node (lu1) [position=60:0.07 from u1] {$u_1$};
		\node (lu2) [position=120:0.07 from u2] {$u_2$};
		\node (lu3) [position=240:0.07 from u3] {$u_3$};
		\node (lu4) [position=300:0.07 from u4] {$u_4$};
		
		\node (lw1) [position=0:0.07cm from w1] {$w_1$};
		\node (lw2) [position=180:0.07cm from w2] {$w_2$};
		
		\node (n) [position=270:2.8cm from anchor3] {$q=2$};
		
		\path
		(u1) edge (u2)
		(u2) edge (w2)
		(w1) edge (u1)
		(u3) edge (u4)
		(u4) edge (w1)
		(w2) edge (u3)
		;

		\end{tikzpicture}
	\end{center}
	\caption{Some additional examples of flowers of size $4$ with $q=2,3,4$.}
	\label{fig3.10}
\end{figure}
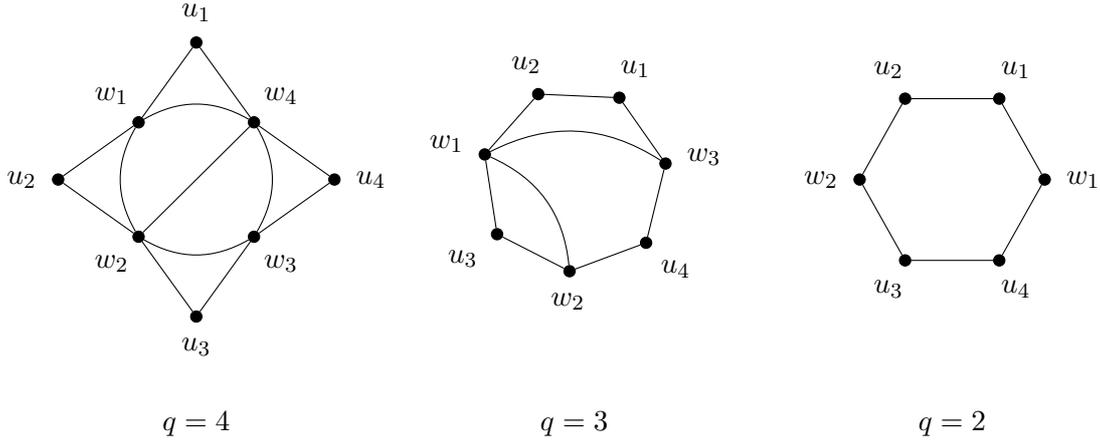

\begin{Theorem}\label{thm:squarechordal}
$G^2$ is chordal if and only if all of its induced flowers of size $n\geq4$ in $G$ are withered.
\end{Theorem}

\begin{proof}
Let $F\in\mathcal{F}_n$ be an unwithered flower of size $n$ in $G$. With $F$ not being withered it follows $\text{dist}_{G}({u_i},{u_j})\geq 3$ for all $j\neq i\pm1\left( \!\!\!\mod n\right)$ and $\text{dist}_{G}({u_i},{u_j})\leq2$ for $j = i\pm 1\left( \!\!\!\mod n\right)$. Hence the $u_i$ form an induced cycle of length $n$ in $G^2$.\\
Now let $C_{G^2}$ be an induced cycle in $G^2$ with vertex sequence $(u_1,\dots,u_n)$. By Lemma \ref{lem:notwoconsec} $G$ can contain at most $\lfloor\frac{n}{2}\rfloor$ edges of $C_{G^2}$. Hence for the other edges $u_iu_{i+1\left( \!\!\!\mod n\right)}$ in $C_{G^2}$ there must exist some vertex $w_k$ with $u_iw_k, u_{i+1\left( \!\!\!\mod n\right)}\in E({G})$ in $G$ additionally satisfying $w_ku_j\notin E({G})$ for all $j\in\{1,\dots n\}\setminus\{u_i,u_{i+1\left( \!\!\!\mod n\right)}\}$. Such vertices $w_i$ exist necessarily for every pair of consecutive vertices of $C_{G^2}$ that are not already adjacent in $G$, i.e., only those necessary $w_i$ are taken into account and there are at least $\lceil{\frac{n}{2}}\rceil$ of them.\\
The $u_i$ together with the $w_k$ form a cycle $C_G$ of length $n+\lceil{\frac{n}{2}}\rceil\leq n+q\leq 2\,n$ in which no two $w_k$ are adjacent. Now for the pending $u$-vertices we have to consider those pairs $w_i$, $w_{i+1}$ with $w_iw_{i+1}\in E({G})$ that are adjacent to a common $u$-vertex. For each such pair of $w$-vertices we shorten the cycle $C_G$ by removing their common neighbor $u$, which thus becomes a pending vertex, from it and adding the edge $w_iw_{i+1}$. The result is the cycle $C$ of condition $i)$ in the definition. No other edge between two $w$-vertices is part of $C$ and no more than two consecutive $u$-vertices appear on it due to the application of Lemma \ref{lem:notwoconsec}, thus $ii)$, $iii)$ and $iv)$ are satisfied as well.\\
In order to prevent $C_{G^2}$ from having a chord every $w_k$ may only be adjacent to those $u$ vertices adjacent to it in $C$ which yields $iv)$ and, besides the edges of $C_{G^2}$ already present in $G,$ no other edges between the vertices in $U$ can exist satisfying condition $v)$ of the definition. Therefore a flower of size $n$ that is not withered exists in $G$. If $F$ were withered $C_{G^2}$ would contain a chord in $G^2$ and therefore would not have been an induced cycle in the first place.
\end{proof}

\section{Chordal Linegraph Squares} \label{Sec:linegraphsquares}

In this section we investigate the chordality of line graph squares. We start with a basic observation:

\begin{Lemma}\label{lemma2.1}
$G$ contains a cycle $C$ of length at least four if and only if $L(G)$ contains an induced cycle $C_L$ of the same length.
\end{Lemma}

\begin{proof}
Suppose that $G$ contains a cycle $C_{G}$. Then, there is also a cycle $C_{L}$ in $L(G).$ If it has a chord there are two nonconsecutive and incident edges of $C_{G}$ which is impossible. \\
Now, suppose that there is an induced cycle $C_{L}$ in $L(G).$ Then, there is a sequence of edges $e_1,\dots,e_n$ with $e_i\cap e_{i+1 \mod n}\neq \emptyset$ and $e_i\cap e_j=\emptyset$ for $j\neq i\pm 1 \mod n$ in $G.$ Therefore, $e_i\cap e_{i+1 \mod n}\neq e_{i-1\mod n}\cap e_{i }$ for $i=1,\dots,n$ and we obtain a cycle $C_{G}$ in $G.$   
\end{proof}

\begin{Lemma}\label{lem:cycinlsquare}
If $G$ does not contain induced cycles of length $l\geq f$, then $L(G)^2$ contains no induced cycles of length $l\geq f$.
\end{Lemma}

\begin{proof}
Let a graph $G$ without induced cycles of length $l\geq f$ be given. Suppose that $L({G})^2$ contains an induced cycle $C_{L^2}$ of length $l\geq f$ with vertex set $u,\dots,u_l$. From Theorem \ref{thm:squarechordal}, we obtain that $L(G)$ contains an unwithered and induced flower $F$ with vertex set $U\cup W$ of the same size. This flower yields a cycle $C_L$ of length $l+q$ in $L(G).$ Chords are only allowed between vertices of the set $W.$ Since $L(G)$ is a line graph, it does not contain any claws. Therefore, the only possible chords connect two consecutive $w\in W$ with exactly one $u$-vertex in between (compare case 1.2 in the proof of Theorem \ref{thm:chordalsq}). Hence, we obtain an induced cycle in $L(G)$ of length $\tilde{l}\geq l$ which yields a cycle $C_G$ of $G$ of the same length consisting of $u$- and $w$-edges. We denote its vertex set by $v_1,\dots,v_{\tilde{l}}.$ Please note, that in $C_G$ at most two $u$-edges are consecutive and every set of consecutive $u$-edges (including the sets of size one) is followed by a $w$-edge. If there are $k$ such "following" $w$-edges the length of $C_G$ is exactly $\tilde{l}=l+k.$   

$C_G$ contains chords by assumption because $G$ has no induced cycles of this length. Fortunately, most types of chords would correspond to a chord of $C_{L^2}$ respectively to a withering of the flower $F,$ namely all chords of the form $v_iv_j$
\begin{itemize}
\item with $|i-j| \mod \tilde{l} \geq 4,$
\item with $|i-j| \mod \tilde{l} = 3,$ if they are incident to edges $u_i$ and $u_j$ with $j\neq i\pm 1 \mod l,$ and 
\item with $|i-j| \mod \tilde{l} = 2,$ if they are incident to edges $u_i$ and $u_j$ with $j\neq i\pm 1 \mod l.$ 
\end{itemize}
The possible chords are depicted in Figure \ref{fig:poschords}. Apparently, every chord reduces the length of $C_G$ by one or two. Since in each situation one or two of the "following" $w$ edges (zig-zag edges) are involved, an induced cycle of length at least $l$ remains in $G$ and we end up with a contradiction.

\begin{figure}
\begin{center}
\begin{tikzpicture}

\node (center) [] {};

\node (a1) [draw,circle,fill,inner sep=1.5pt,position=270:0.1cm from center] {};
\node (a2) [draw,circle,fill,inner sep=1.5pt,position=180:1cm from a1] {};
\node (a3) [draw,circle,fill,inner sep=1.5pt,position=180:2cm from a1] {};
\node (a4) [draw,circle,fill,inner sep=1.5pt,position=180:3cm from a1] {};
\node (a5) [draw,circle,fill,inner sep=1.5pt,position=180:4cm from a1] {};
\node (a6) [draw,circle,fill,inner sep=1.5pt,position=180:5cm from a1] {};
%\node (a7) [draw,circle,fill,inner sep=1.5pt,position=180:6cm from a1] {};

\node (l1) [position=210:0.18cm from a1] {$w$};
\node (l2) [position=210:0.18cm from a2] {$u$};
\node (l3) [position=210:0.18cm from a3] {$w$};
\node (l4) [position=210:0.18cm from a4] {$u$};
\node (l5) [position=210:0.18cm from a5] {$w$};
%\node (l6) [position=210:0.18cm from a6] {$u$};

\path[>=latex]
(a1) edge[->] (a2)
(a2) edge[->] (a3)
(a3) edge[decoration = {zigzag,segment length = 1mm, amplitude = 1mm}, decorate]  (a4)
(a4) edge[->] (a5)
(a5) edge[decoration = {zigzag,segment length = 1mm, amplitude = 1mm}, decorate]  (a6)
%(a6) edge (a7)
;

\path
(a2) edge [bend right,dotted] (a5)
;

\end{tikzpicture}

\begin{tikzpicture}

\node (center) [] {};

%\node (a1) [draw,circle,fill,inner sep=1.5pt,position=270:0.1cm from center] {};
\node (a2) [draw,circle,fill,inner sep=1.5pt,position=180:1cm from a1] {};
\node (a3) [draw,circle,fill,inner sep=1.5pt,position=180:2cm from a1] {};
\node (a4) [draw,circle,fill,inner sep=1.5pt,position=180:3cm from a1] {};
\node (a5) [draw,circle,fill,inner sep=1.5pt,position=180:4cm from a1] {};
\node (a6) [draw,circle,fill,inner sep=1.5pt,position=180:5cm from a1] {};
%\node (a7) [draw,circle,fill,inner sep=1.5pt,position=180:6cm from a1] {};

%\node (l1) [position=210:0.18cm from a1] {$w$};
\node (l3) [position=210:0.18cm from a3] {$w$};
\node (l4) [position=210:0.18cm from a4] {$u$};
\node (l5) [position=210:0.18cm from a5] {$w$};
%\node (l6) [position=210:0.18cm from a6] {$u$};

\path[>=latex]
%(a1) edge (a2)
(a2) edge[->] node[below] {$u/w$} (a3) 
(a3) edge[->] (a4)
(a4) edge[->] (a5)
(a5) edge[decoration = {zigzag,segment length = 1mm, amplitude = 1mm}, decorate]  (a6)
%(a6) edge (a7)
;

\path
(a3) edge [bend right,dotted] (a5)
;

\end{tikzpicture}

\begin{tikzpicture}

\node (center) [] {};

%\node (a1) [draw,circle,fill,inner sep=1.5pt,position=270:0.1cm from center] {};
\node (a2) [draw,circle,fill,inner sep=1.5pt,position=180:1cm from a1] {};
\node (a3) [draw,circle,fill,inner sep=1.5pt,position=180:2cm from a1] {};
\node (a4) [draw,circle,fill,inner sep=1.5pt,position=180:3cm from a1] {};
\node (a5) [draw,circle,fill,inner sep=1.5pt,position=180:4cm from a1] {};
\node (a6) [draw,circle,fill,inner sep=1.5pt,position=180:5cm from a1] {};
%\node (a7) [draw,circle,fill,inner sep=1.5pt,position=180:6cm from a1] {};

\node (l2) [position=210:0.18cm from a2] {$w$};
\node (l3) [position=210:0.18cm from a3] {$u$};
\node (l4) [position=210:0.18cm from a4] {$u$};
\node (l5) [position=210:0.18cm from a5] {$w$};
%\node (l6) [position=210:0.18cm from a6] {$u$};

\path[>=latex]
%(a1) edge (a2)
(a2) edge[->] (a3) 
(a3) edge[->] (a4)
(a4) edge[->] (a5)
(a5) edge[decoration = {zigzag,segment length = 1mm, amplitude = 1mm}, decorate]  (a6)
%(a6) edge (a7)
;

\path
(a3) edge [bend right,dotted] (a5)
;

\end{tikzpicture}

\end{center}
\caption{Possible chords in $C_G$.}
\label{fig:poschords}
\end{figure}
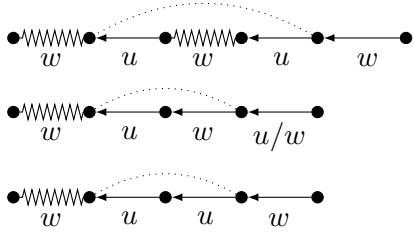

\end{proof}

In the case $f=4$, the foregoing Lemma yields a theorem due to Cameron.

\begin{Theorem}\cite{cameron1989induced}\label{thm:lgsquare}
Let $G$ be chordal, then $L(G)^2$ is chordal.
\end{Theorem}

In the following, we characterize the class of graphs with chordal line graph squares. Motivated by the definition of flowers we define sprouts:

\begin{Definition}[Sprout]
A {\em sprout} of size $n$ is a graph $S=\left( V,U\cup W\cup E\right)$ with $|{U}|=n$ and $|{W}|=q$, $U$, $W$ and $E$ pairwise disjoint and $\lceil{\frac{n}{2}}\rceil\leq q\leq n$ satisfying the following conditions:
\begin{enumerate}[i)]

\item There is a cycle $C$ with $E({C})\supseteq W$ containing the edges of $W$ in the order $w_1,\dots,w_q$.

\item The set $U=\{u_1,\dots,u_n\}$ is sorted by the appearance order of its elements along $C$ with $u_1\cap w_q\neq\emptyset$ and $u_2\cap w_1\neq\emptyset$, in addition $u_i\cap u_j=\emptyset$ for $j\neq i\pm 1\left( \!\!\!\mod n\right)$.

\item If $w_i\cap w_{i+1}\neq\emptyset$, then there is exactly one $u\in U$ with $w_i\cap w_{i+1}\cap u\neq\emptyset$, those edges are called {\em pending}.

\item If $w_i\cap w_{i+1}=\emptyset$, then there either is one $u\in U$ connecting $w_i$ and $w_{i+1}$ in $C$, or there are exactly two edges $t,u\in U$, such that the path $w_ituw_{i+1}$ is part of $C$.

\item The pending $u$-edges are pairwise nonadjacent and all $u$-edges that are not pending are edges of $C$.
\end{enumerate}
The family of all sprouts of size $n$ is denoted by $\mathfrak{S}_n$.\\
If a sprout $S$ contains an edge $e\in E$ connecting two non-consecutive $u$-edges, we say $S$ is {\em infertile}, otherwise $S$ is called {\em fertile}.
\end{Definition}

For some examples of sprouts compare Figure \ref{fig.16}.

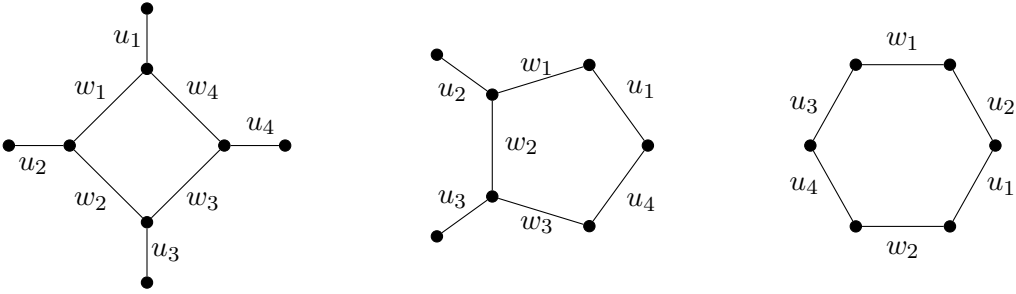
\begin{figure}[!h]
\begin{center}
\begin{tikzpicture}

\node (anchor1) [] {};
\node (anchor2) [position=0:5.2cm from anchor1] {};
\node (anchor3) [position=0:4.2cm from anchor2] {};

\node (u1) [draw,circle,fill,inner sep=1.5pt,position=90:1.6cm from anchor1] {};
\node (u2) [draw,circle,fill,inner sep=1.5pt,position=180:1.6cm from anchor1] {};
\node (u3) [draw,circle,fill,inner sep=1.5pt,position=270:1.6cm from anchor1] {};
\node (u4) [draw,circle,fill,inner sep=1.5pt,position=0:1.6cm from anchor1] {};

\node (w1) [draw,circle,fill,inner sep=1.5pt,position=90:0.8cm from anchor1] {};
\node (w2) [draw,circle,fill,inner sep=1.5pt,position=180:0.8cm from anchor1] {};
\node (w3) [draw,circle,fill,inner sep=1.5pt,position=270:0.8cm from anchor1] {};
\node (w4) [draw,circle,fill,inner sep=1.5pt,position=0:0.8cm from anchor1] {};

\node (lu1) [position=100:1.05 from anchor1] {$u_1$};
\node (lu2) [position=190:1.05 from anchor1] {$u_2$};
\node (lu3) [position=280:1.05 from anchor1] {$u_3$};
\node (lu4) [position=10:1.05 from anchor1] {$u_4$};

\node (lw1) [position=135:0.5cm from anchor1] {$w_1$};
\node (lw2) [position=225:0.5cm from anchor1] {$w_2$};
\node (lw3) [position=315:0.5cm from anchor1] {$w_3$};
\node (lw4) [position=45:0.5cm from anchor1] {$w_4$};

\path
(u1) edge (w1)
(u2) edge (w2)
(u3) edge (w3)
(u4) edge (w4)
;

\path
(w1) edge (w2)
(w2) edge (w3)
(w3) edge (w4)
(w4) edge (w1)
;

% % % % % % % % % % % % % % %

%Base Circle C_5

\node (v1) [draw,circle,fill,inner sep=1.5pt,position=0:0.9cm from anchor2] {};
\node (v2) [draw,circle,fill,inner sep=1.5pt,position=72:0.9cm from anchor2] {};
\node (v3) [draw,circle,fill,inner sep=1.5pt,position=144:0.9cm from anchor2] {};
\node (v4) [draw,circle,fill,inner sep=1.5pt,position=216:0.9cm from anchor2] {};
\node (v5) [draw,circle,fill,inner sep=1.5pt,position=288:0.9cm from anchor2] {};

\path
(v1) edge (v2)
(v2) edge (v3)
(v3) edge (v4)
(v4) edge (v5)
(v5) edge (v1)
;

\node (cl1) [position=36:0.7cm from anchor2] {$u_1$};
\node (cl2) [position=108:0.7cm from anchor2] {$w_1$};
\node (cl3) [position=180:0.05cm from anchor2] {$w_2$};
\node (cl4) [position=252:0.7cm from anchor2] {$w_3$};
\node (cl5) [position=324:0.7cm from anchor2] {$u_4$};

%Additional Nodes

\node (a1) [draw,circle,fill,inner sep=1.5pt,position=144:1.8cm from anchor2] {};
\node (a2) [draw,circle,fill,inner sep=1.5pt,position=216:1.8cm from anchor2] {};

\path
(v3) edge (a1)
(v4) edge (a2)
;

\node (al1) [position=154:1.1cm from anchor2] {$u_2$};
\node (al2) [position=206:1.1cm from anchor2] {$u_3$};

% % % % % % % % % % % % % % %

\node (u1) [draw,circle,fill,inner sep=1.5pt,position=60:1cm from anchor3] {};
\node (u2) [draw,circle,fill,inner sep=1.5pt,position=120:1cm from anchor3] {};
\node (u3) [draw,circle,fill,inner sep=1.5pt,position=240:1cm from anchor3] {};
\node (u4) [draw,circle,fill,inner sep=1.5pt,position=300:1cm from anchor3] {};

\node (w1) [draw,circle,fill,inner sep=1.5pt,position=0:1cm from anchor3] {};
\node (w2) [draw,circle,fill,inner sep=1.5pt,position=180:1cm from anchor3] {};

\node (lu1) [position=337.5:0.9cm from anchor3] {$u_1$};
\node (lu2) [position=22.5:0.9cm from anchor3] {$u_2$};
\node (lu3) [position=157.5:0.9cm from anchor3] {$u_3$};
\node (lu4) [position=202.5:0.9cm from anchor3] {$u_4$};

\node (lw1) [position=90:1cm from anchor3] {$w_1$};
\node (lw2) [position=270:1cm from anchor3] {$w_2$};

\path
(u1) edge (u2)
(u2) edge (w2)
(w1) edge (u1)
(u3) edge (u4)
(u4) edge (w1)
(w2) edge (u3)
;

\end{tikzpicture}
\end{center}
\caption{Examples of sprouts of size $4$.}
\label{fig.16}
\end{figure}

\begin{Lemma}\label{lem:cycsprout}
$L(G)^2$ contains an induced cycle of length $l$ if and only if $G$ contains an unwithered sprout of size $l.$
\end{Lemma}

\begin{proof}
We show that the existence of a fertile sprout in $G$ is equivalent to the existence of an induced and unwithered flower of the same size in $L(G).$ Then, the assertion follows directly from Theorem \ref{thm:squarechordal}.

Let $S$ be a fertile sprout with edge classes $U$, $W$ and $E$ in $G$, furthermore let $C_G$ be the cycle of the sprout definition containing all $w$-edges. This cycle may contain chords $w\in E$, which again are of the three types seen in Figure \ref{fig:poschords}. Lemma \ref{lemma2.1} yields the existence of an induced cycle $C_L'$ in $L({G})$ whose vertices correspond to the edges of $C_G$, which are completely contained in $U\cup W$.\\
Any edge $u\in U$ which is not in $C_G$ is pending and therefore forms a star together with its two adjacent $w$-edges. Hence in $L({G})$ we obtain a triangle. We choose the cycle $C_L'$ as the cycle required in $i)$ of the flower definition. With any pending edge ending up as a $u$-vertex in the line graph which is adjacent to exactly two $w$-vertices that are adjacent conditions $ii)$, $iii)$ and $v)$ are satisfied as well. At last condition $iv)$ of the sprout definition guarantees that a $u$-vertex in $C_L'$ either is adjacent to no other $u$-vertex, or to exactly one. This corresponds to condition $iv)$ of the flower definition which therefore is also satisfied. So we obtain a flower $F$ with size $|{U}|$ in $L({G})$. Any vertex in $L({G})$ responsible for $F$ being withered would correspond to an edge $e\in E$, rendering the sprout infertile.\\
So now suppose there is a non-withered flower $F$ in $L({G})$. We define $C_L'$ to be an induced cycle in $F$ containing as many $w$-vertices as possible. If not all $w$-vertices lie on $C_L'$, each cycle (for example the cycle of the flower definition containing all $u$- and $w$-vertices except the pending ones) covering all $w$-vertices contains a chord. Such a chord joins two consecutive $u$-vertices that both are not adjacent to another $u$-vertex. Otherwise it contradicts either the definition of a flower or that $L({G})$ is a line graph by generating an induced $K_{1,3}$ (compare the proof of Theorem \ref{lem:cycinlsquare}).\\
Now for the $u$-$u$ chord. If there is another chord joining the skipped $w$-vertex to another, consecutive $w$-vertex, the skipped one can be included in an induced cycle containing more $w$-vertices contradicting the maximality of $C_L'$. If there is no such chord we are in the case depicted in Figure \autoref{fig4.15}.
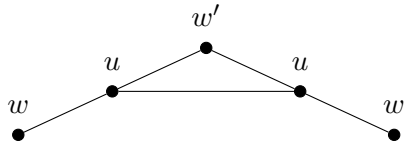
\begin{figure}
	\begin{center}
		\begin{tikzpicture}
		\node (v1) [draw,circle,fill,inner sep=1.5pt] {};
		\node (v2) [draw,circle,fill,inner sep=1.5pt,position=25:1.2cm from v1] {};
		\node (v3) [draw,circle,fill,inner sep=1.5pt,position=25:1.2cm from v2] {};
		\node (v4) [draw,circle,fill,inner sep=1.5pt,position=335:1.2cm from v3] {};
		\node (v5) [draw,circle,fill,inner sep=1.5pt,position=335:1.2cm from v4] {};
		
		\node (l1) [position=90:0.07 from v1] {$w$};
		\node (l2) [position=90:0.07 from v2] {$u$};
		\node (l3) [position=90:0.07 from v3] {$w'$};
		\node (l4) [position=90:0.07 from v4] {$u$};
		\node (l5) [position=90:0.07 from v5] {$w$};

		\path
		(v1) edge (v2)
		(v2) edge (v3)
		(v3) edge (v4)
		(v4) edge (v5)
		;
		
		\path 
		(v2) edge (v4)
		;
		
		\end{tikzpicture}
	\end{center}
	\caption{The sole legal chord excluding a $w$-vertex from $C_L'$.}
	\label{fig4.15}
\end{figure}
In this case the excluded $w$-vertex $w'$ is not necessary and $F'=\left( U\cup W\setminus \{w'\}, E'\right)$ is a flower of the same size. Hence all $w'$-vertices lie on $C_L'$ and Lemma \ref{lemma2.1} yields the existence of a cycle $C_G$ in $G$ completely consisting of all edges corresponding to the $w$-vertices in $L(G)$ and some $u$-vertices. Hence condition $i)$ of the sprout definition is satisfied.\\
The remaining $u$-vertices cannot be put into the cycle and thus they correspond to the pending edges. Hence we obtain a sprout in $G$ which is fertile since any violation of the other conditions would result in $F$ not being a flower. 
\end{proof}

\begin{Theorem}\label{thm:linegsquarechordal}
$L(G)^2$ is chordal if and only  if  $G$ does not contain induced cycles of length $l\geq 6$ and no fertile sprouts of size 4 and 5.
\end{Theorem}

\begin{proof}
From Lemma \ref{lem:cycinlsquare}, we know that forbidden induced cycles in $G$ lead to forbidden cycles in $L(G)^2.$ Therefore, if induced cycles of length $l\geq 6$ and unwithered sprouts of size 4 and 5 are forbidden, then $L(G)^2$ will be chordal.

If $L(G)^2$ is chordal, by Lemma \ref{lem:cycsprout} $G$ contains only infertile sprouts. In particular, it does not contain fertile sprouts of size 4 and 5 and induced cycles of greater length (since they are special sprouts). 
\end{proof}

\section*{Conclusion}
We investigated the chordality of graphs and line graph squares. Thereby, we gave a sufficient condition for the chordality of the square of a graph without induced cycles of length at least five. Up to our knowledge only the two sufficient conditions mentioned in Section \ref{Sec:chordalsquares} were known before.\\
The work of Laskar and Shier on squares of chordal graphs let us to the graph class of sunflowers. We generalized this concept to the graph class of flowers and used the generalization to characterize the chordality of graph squares at the end of Section \ref{Sec:chordalsquares}. By translating flowers to line graph we obtained the class of sprouts. Since line graphs contain no claws we were able to simplify the definitions slightly compared to flowers. \\    
The last result of our work, the characterization of the chordality of line graphs, can be investigated further. An investigation of sprouts of size four and five leads to the following result which we just state without proof:

\begin{Theorem}\label{thm4.15}
	 $L({G})^2$ is chordal if and only if $G$ does not contain an induced cycle of length $l\geq6$ or an induced subgraph as shown in Figure \ref{fig4.39} (the gray, dotted edges may or may not exist). 	
\end{Theorem} 

\begin{figure}[!h]
	\begin{center}
		\begin{tikzpicture}
		\node (anchor1) [] {};
		%\node (i) [position=270:2.25cm from anchor1] {I};
		
		\node (anchor4) [position=0:5.2cm from anchor1] {};
		%\node (iv) [position=270:2.25cm from anchor4] {III$_a$};
		
		% % % % % % % % % % % % % % % % % % % % % %
		%Anchor 1
		% % % % % % % % % % % % % % % % % % % % % %
		
		%Base Circle C_5
		
		\node (v1) [draw,circle,fill,inner sep=1.5pt,position=0:0.9cm from anchor1] {};
		\node (v2) [draw,circle,fill,inner sep=1.5pt,position=72:0.9cm from anchor1] {};
		\node (v3) [draw,circle,fill,inner sep=1.5pt,position=144:0.9cm from anchor1] {};
		\node (v4) [draw,circle,fill,inner sep=1.5pt,position=216:0.9cm from anchor1] {};
		\node (v5) [draw,circle,fill,inner sep=1.5pt,position=288:0.9cm from anchor1] {};
		
		\path
		(v1) edge (v2)
		(v2) edge (v3)
		(v3) edge (v4)
		(v4) edge (v5)
		(v5) edge (v1)
		;
		
		\node (cl1) [position=36:0.7cm from anchor1] {};
		\node (cl2) [position=108:2mm from anchor1] {};
		\node (cl3) [position=180:0.05cm from anchor1] {};
		\node (cl4) [position=252:2mm from anchor1] {};
		\node (cl5) [position=324:0.7cm from anchor1] {};
		
		%Additional Nodes
		
		\node (a1) [draw,circle,fill,inner sep=1.5pt,position=144:1.8cm from anchor1] {};
		\node (a2) [draw,circle,fill,inner sep=1.5pt,position=216:1.8cm from anchor1] {};
		
		\path
		(v3) edge (a1)
		(v4) edge (a2)
		;
		
		\node (al1) [position=144:2cm from anchor1] {};
		\node (al2) [position=216:2cm from anchor1] {};
		
		%Possible Edges in E
		
		\path[color=lightgray,thick,dotted]
%		(a1) edge [bend left] (v2)
%		edge (v4)
%		(a2) edge (v3)
%		(a2) edge [bend right] (v5)
		(a1) edge [bend right] (a2)
		;

		% % % % % % % % % % % % % % % % % % % % % % %
		%Anchor4
		% % % % % % % % % % % % % % % % % % % % % % % 
		
		\node (u1) [draw,circle,fill,inner sep=1.5pt,position=90:1.1cm from anchor4] {};
		\node (u2) [draw,circle,fill,inner sep=1.5pt,position=150:1.1cm from anchor4] {};
		\node (u3) [draw,circle,fill,inner sep=1.5pt,position=210:1.1cm from anchor4] {};
		\node (u4) [draw,circle,fill,inner sep=1.5pt,position=270:1.1cm from anchor4] {};
		\node (u5) [draw,circle,fill,inner sep=1.5pt,position=330:1.1cm from anchor4] {};
		\node (u6) [draw,circle,fill,inner sep=1.5pt,position=30:1.1cm from anchor4] {};
		
		\node (l1) [position=120:1cm from anchor4] {};
		\node (l2) [position=180.5:1cm from anchor4] {};
		\node (l3) [position=240.5:1cm from anchor4] {};
		\node (l4) [position=300.5:1cm from anchor4] {};
		\node (l5) [position=0.5:1cm from anchor4] {};
		\node (l6) [position=60.5:1cm from anchor4] {};
		
		\path
		(u1) edge (u2)
		(u2) edge (u3)
		(u3) edge (u4)
		(u4) edge (u5)
		(u5) edge (u6)
		(u6) edge (u1)
		;
		
		\path[thick]
		(u2) edge (u6)
		;
		
		\end{tikzpicture}
	
		\begin{tikzpicture}
		
		\node (anchor1) [] {};
		\node (anchor2) [position=0:5.2cm from anchor1] {};
		\node (anchor3) [position=0:5.2cm from anchor2] {};
		
		\node (v1) [inner sep=1.5pt,position=135:0.6cm from anchor1,draw,circle] {};
		\node (v2) [inner sep=1.5pt,position=225:0.6cm from anchor1,draw,circle] {};
		\node (v3) [inner sep=1.5pt,position=315:0.6cm from anchor1,draw,circle] {};
		\node (v4) [inner sep=1.5pt,position=45:0.6cm from anchor1,draw,circle] {};
		
		\path
		(v1) edge (v2)
		(v2) edge (v3) 
		(v3) edge (v4)
		(v4) edge (v1)
		;
		
		\node (u1) [inner sep=1.5pt,position=135:1.6cm from anchor1,draw,circle,fill] {};
		\node (u2) [inner sep=1.5pt,position=225:1.6cm from anchor1,draw,circle,fill] {};
		\node (u3) [inner sep=1.5pt,position=315:1.6cm from anchor1,draw,circle,fill] {};
		\node (u4) [inner sep=1.5pt,position=45:1.6cm from anchor1,draw,circle,fill] {};
		
		\path
		(u1) edge (v1)
		(u2) edge (v2)
		(u3) edge (v3)
		(u4) edge (v4)
		;
		
		\path[color=lightgray, thick,dotted]
		(u1) edge (u2)
		(u2) edge (u3)
		(u3) edge (u4)
		(u4) edge (u1)
		
%		(u1) edge (v2)
%		edge (v4)
%		(u2) edge (v3)
%		edge (v1)
%		(u3) edge (v4)
%		edge (v2)
%		(u4) edge (v1)
%		edge (v3)
		;
		
		\node (lu1) [position=135:0.07 from u1] {};
		\node (lu2) [position=225:0.07 from u2] {};
		\node (lu3) [position=315:0.07 from u3] {};
		\node (lu4) [position=45:0.07 from u4] {};
		
		% % % % % % % % % % % % % % %
		
		\node (v1) [inner sep=1.5pt,position=135:0.6cm from anchor2,draw,circle] {};
		\node (v2) [inner sep=1.5pt,position=225:0.6cm from anchor2,draw,circle] {};
		\node (v3) [inner sep=1.5pt,position=315:0.6cm from anchor2,draw,circle] {};
		\node (v4) [inner sep=1.5pt,position=45:0.6cm from anchor2,draw,circle] {};
		
		\path
		(v1) edge (v2)
		(v2) edge (v3)
		(v3) edge (v4)
		(v4) edge (v1)
		;
		
		\node (u1) [inner sep=1.5pt,position=135:1.6cm from anchor2,draw,circle,fill] {};
		\node (u2) [inner sep=1.5pt,position=225:1.6cm from anchor2,draw,circle,fill] {};
		\node (u3) [inner sep=1.5pt,position=0:1.6cm from anchor2,draw,circle,fill] {};
		\node (u4) [inner sep=1.5pt,position=0:1.6cm from anchor2,draw,circle,fill] {};
		
		\path
		(u1) edge (v1)
		(u2) edge (v2)
		(u3) edge (v3)
		(u4) edge (v4)
		;
		
		\path[color=lightgray, thick,dotted]
		(u1) edge (u2)
		
%		(u1) edge (v2)
		(u1) edge (v4)
		(u2) edge (v3)
%		(u2) edge (v1)
		;
		
		\node (lu1) [position=135:0.07 from u1] {};
		\node (lu2) [position=225:0.07 from u2] {};
		\node (lu3) [position=270:0.3 from u3] {};

		% % % % % % % % % % % % % % %
		
		\node (v1) [inner sep=1.5pt,position=135:0.6cm from anchor3,draw,circle] {};
		\node (v2) [inner sep=1.5pt,position=225:0.6cm from anchor3,draw,circle] {};
		\node (v3) [inner sep=1.5pt,position=315:0.6cm from anchor3,draw,circle] {};
		\node (v4) [inner sep=1.5pt,position=45:0.6cm from anchor3,draw,circle] {};
		
		\path
		(v1) edge (v2)
		(v2) edge (v3)
		(v3) edge (v4)
		(v4) edge (v1)
		;
		
		\node (u1) [inner sep=1.5pt,position=180:1.6cm from anchor3,draw,circle,fill] {};
		\node (u2) [inner sep=1.5pt,position=180:1.6cm from anchor3,draw,circle,fill] {};
		\node (u3) [inner sep=1.5pt,position=0:1.6cm from anchor3,draw,circle,fill] {};
		\node (u4) [inner sep=1.5pt,position=0:1.6cm from anchor3,draw,circle,fill] {};
		
		\path
		(u1) edge (v1)
		(u2) edge (v2)
		(u3) edge (v3)
		(u4) edge (v4)
		;
		
		\node (lu1) [position=90:0.3 from u1] {};
		\node (lu3) [position=270:0.3 from u3] {};

		\end{tikzpicture}
		\vspace{-1mm}
	\end{center}
	\caption{Fives types of forbidden subgraphs.}
	\label{fig4.39}
\end{figure}
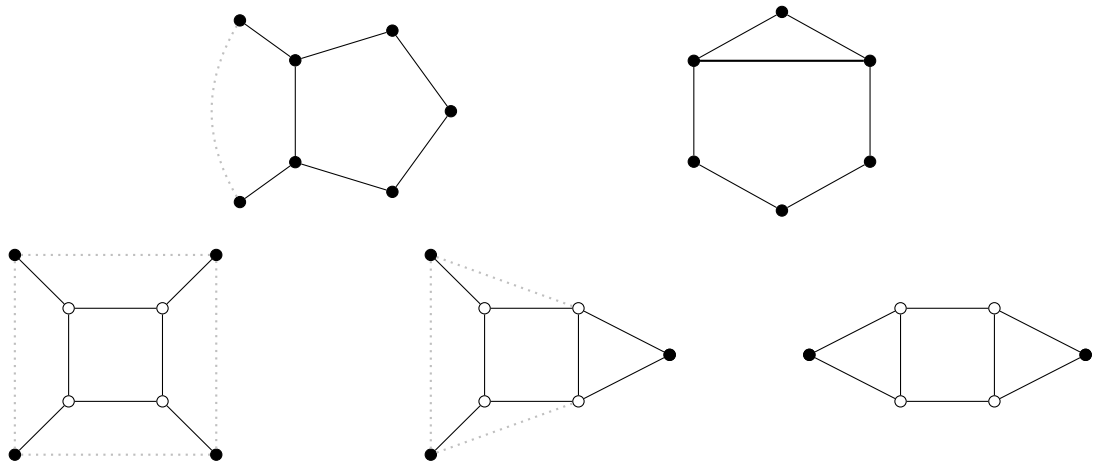

\bibliographystyle{plain}
\bibliography{bibliography}

\end{document}